\documentclass{article}

\usepackage{arxiv}
\usepackage{amsmath}
\usepackage{amssymb}
 \usepackage{amscd}

\usepackage[utf8]{inputenc} 
\usepackage[T1]{fontenc}    
\usepackage{url}            
\usepackage{booktabs}       
\usepackage{nicefrac}       
\usepackage{microtype}      
\usepackage{graphicx}
\usepackage{doi}
\usepackage[english]{babel}
\usepackage{graphicx,epstopdf,epsfig}
\usepackage{amsfonts,epsfig,fancyhdr,graphics, hyperref,amsmath,amssymb}
\usepackage{tikz}
\usetikzlibrary{arrows.meta, positioning}
\usetikzlibrary{arrows.meta, positioning, calc, fit}
\usetikzlibrary{arrows.meta, positioning, fit}

\newtheorem{theorem}{Theorem}[section]

\newtheorem{lemma}[theorem]{Lemma}

\newtheorem{proposition}[theorem]{Proposition}
\newtheorem{remark}[theorem]{Remark}

\newenvironment{proof}[1][Proof]{\textbf{#1.} }{\ \rule{0.5em}{0.5em}}
\newtheorem{example}[theorem]{Example}
\newtheorem{question}{Question}

\AtEndEnvironment{example}{\hfill \ensuremath{\blacksquare}}
\AtEndEnvironment{remark}{\hfill \ensuremath{\blacksquare}}

\newcommand{\Path}{\mathcal P}
\newcommand{\diag}{\operatorname{diag}}
\newcommand{\rank}{\operatorname{rank}}
\newcommand{\termrank}{\operatorname{termrank}}
\newcommand{\spec}{\operatorname{spec}}

\newcommand{\numinus}{\nu_-}
\newcommand{\ones}{\mathbf{1}}

\author{ Bassam Mourad \\
	Department of Mathematics,\\  Faculty of Science, Lebanese University,\\ Beirut, Lebanon\\
	\texttt{ bmourad@ul.edu.lb} \\
    \And
     Issam Kaddoura \\
	Department of Mathematics,\\ Lebanese International University,\\  Saida, Lebanon\\ 
	\texttt{issam.kaddoura@liu.edu.lb} \\
	\And
	Hassan Issa \\
	Department of Mathematics,\\  Faculty of Science, Lebanese University,\\ Beirut, Lebanon\\
	\texttt{ hassan.issa@ul.edu.lb} \\
   }

\begin{document}

\title{ Negative index, matchings, and nonnegative eigenvalues of tridiagonal stochastic matrices }

\maketitle

\begin{abstract}
We study negative eigenvalues of $n\times n$ stochastic matrices whose off-diagonal support is constrained by a sparse graph.  The main tool is a matching-based inertia principle: if $G$ is bipartite with  matching number $\mu(G)$, $S$ is a real symmetric matrix supported on $G$ with nonnegative  diagonal entries  and whose  negative  index (i.e. number of negative eigenvalues counted with their multiplicities) is denoted by \( \numinus(S) \),  then
\[
        \numinus(S)\leq \mu(G).
\]
  In particular, every $n\times n$ nonnegative tridiagonal stochastic matrix $P$ satisfies
\(
        \numinus(P)\leq \left\lfloor \frac n2\right\rfloor.
\)
Consequently, after ordering the eigenvalues of $P$  in the decreasing order, we have
\(
        \lambda_{\lceil n/2\rceil}(P)\geq0,
        \  \text{and hence} \ 
        \lambda_2(P)\geq0,  \mbox{ for } n\geq3.
\)
This gives an all-dimensional strengthening of the previously known $4\times4$ tridiagonal stochastic result.  Next, we show that this tridiagonal bound is sharp in every dimension in both reducible and irreducible cases. Finally,  we explore some possible extension and raise some open questions.
\end{abstract}

\textbf{AMS CLASSIFICATION}: 15A18, 15B51, 05C50, 60J10.




\section{Introduction}

A row-stochastic (stochastic for short) matrix  $P=(p_{ij})$ is a nonnegative matrix whose each row sum is equal to 1. Such a matrix  is the transition matrix of a finite Markov chain where its eigenvalues encode convergence, periodicity and oscillatory behaviour (see, e. g., \cite{LevinPeresWilmer2017}).  The present paper is concerned with the negative part of the spectrum when the  transition matrices are  supported on  a bipartite graph.

The motivating case is that of $n\times n$  tridiagonal stochastic matrices,
\[
P=
\begin{pmatrix}
 p_{11} & p_{12} & 0 & \cdots & 0 \\
 p_{21} & p_{22} & p_{23} & \ddots & \vdots \\
 0 & p_{32} & p_{33} & \ddots & 0 \\
 \vdots & \ddots & \ddots & \ddots & p_{n-1,n} \\
 0 & \cdots & 0 & p_{n,n-1} & p_{nn}
\end{pmatrix},
\  P \mbox{ is nonnegative}  \mbox{ and }
P\ones=\ones,
\] where $\ones$ is the all-ones vector.
These are precisely the transition matrices of finite birth-death chains on the path graph $\Path_n$  on $n$ vertices.  

It is well known (see, e.g., \cite{GolubVanLoan2013}) that for an  $n\times n$   tridiagonal matrix to have completely real eigenvalues, two requirements must be met: all main diagonal elements must be real and the product of every two opposite  elements in the subdiagonal and superdiagonal must be non-negative.    Thus, any stochastic tridiagonal matrix has real eigenvalues and so they can be ordered  nonincreasingly as $1=\lambda_1\geq \lambda_2\geq \ldots\geq\lambda_n\geq -1$ (by the Perron-Frobenius theorem \cite{BermanPlemmons1994}).
 
Tridiagonal stochastic matrices are the finite-state, discrete-time analogue of
birth-death dynamics: the chain moves only between neighbouring states, as in
queue-length, population-size, inventory-level, and path random-walk models.
The spectral theory of such one-dimensional Markov chains is classical and is
closely related to three-term recurrences and orthogonal polynomials
\cite{Dominguez2021,KarlinMcGregor1957}.

The motivation for this work is to understand how the sparsity pattern of a reversible transition matrix
limits the dimension of its negative spectral subspace which is the subspace spanned by all eigenvectors corresponding to negative eigenvalues.   
 In this context, our main results show that  the negative part of the spectrum is
controlled not only by the number of states, but also by a concrete combinatorial invariant of the allowed transition graph. In other words, the result captures how the geometry of the support graph limits the
number of independent negative, or oscillatory, spectral modes. In addition,  from the Markov-chain point of view, the second largest eigenvalue is a central  
 spectral parameter:
besides the stationary eigenvalue \(1\), it is the largest algebraic
nonstationary eigenvalue and hence  it controls the slowest nontrivial mode of
relaxation toward equilibrium in the reversible setting. 
 For a reversible transition matrix \(P\), the detailed-balance condition implies
that \(P\) is similar to a real symmetric matrix, and hence all eigenvalues of
\(P\) are real \cite{AldousFill2002,LevinPeresWilmer2017}.  Negative eigenvalues correspond to alternating spectral modes:
if \(Pv=\lambda v\) with \(\lambda<0\), then
\[
        P^t v=\lambda^t v=(-1)^t|\lambda|^t v,
\]
 after \(t\)
steps, and therefore alternates sign from one step to the next \cite{LevinPeresWilmer2017}.  Such alternating modes are particularly natural for bipartite
support graphs, where transitions occur between two vertex classes.

    It is therefore natural to ask whether the nearest-neighbour structure of a
tridiagonal stochastic matrix forces the dominant nonstationary algebraic mode,
namely the second largest eigenvalue, to be nonnegative.  This leads to the following question which was raised in  \cite{VagenendeVerbekenGuerry2026}.
\begin{question}\label{q1:q}
For an $n\times n$  tridiagonal stochastic matrix, must the second largest eigenvalue be nonnegative?
\end{question}

The case $n=4$ was recently settled in \cite{VagenendeVerbekenGuerry2026} who resolved a conjecture in  \cite{RanTeng2024}, stating that the second largest eigenvalue of every $4\times4$ tridiagonal stochastic matrix is nonnegative.  The conjecture arose in the study of eigenvalue regions for stochastic matrices with prescribed zero patterns.  the authors in \cite{VagenendeVerbekenGuerry2026} proved the conjectured nonnegativity of $\lambda_2$ for irreducible $4\times4$ tridiagonal stochastic matrices and then extended it to the reducible $4\times4$ case. It is worthy to point out that
their proof is tied to a low-dimensional eigenvalue-region analysis.  The present approach is different in the sense that  it identifies a graph-theoretic obstruction to negative eigenvalues and thereby gives a  proof for any dimension $n$.

In particular,  we answer Question \ref{q1:q}  in more general context. In fact, we prove a sharper inertia result stated in the following theorem. 

\begin{theorem}\label{thm:intro-main}
Let $P$ be an $n\times n$  tridiagonal stochastic matrix and let $ \numinus(P)$ denote the number of negative  eigenvalues of $P$. Then,
\[
        \numinus(P)\leq \left\lfloor \frac n2\right\rfloor.
\]
Consequently, after ordering the eigenvalues as
\(
       1= \lambda_1(P)\geq\lambda_2(P)\geq\cdots\geq\lambda_n(P),
\)
one has
\(
        \lambda_{\lceil n/2\rceil}(P)\geq0.
\)
In particular, $\lambda_2(P)\geq0$ for every $n\geq3$.
\end{theorem}

The integer $\lfloor n/2\rfloor$ is not accidental.  It is the matching number of the path graph $\Path_n$.  Noticing that the path graph is bipartite, then we show that the preceding theorem can be obtained via symmetrization from  the following   graph-theoretic result which we prove below.
 \begin{theorem} if $G$ is bipartite and $S$ is real symmetric, supported on $G$, with nonnegative diagonal entries, then $\numinus(S)\leq \mu(G)$ where   $ \mu(G)$ is matching number of $G$.
\end{theorem}


The ingredients in the proofs are classical: Weyl monotonicity for Hermitian eigenvalues, the spectral structure of bipartite block matrices, and the rank--term-rank bound for matrices with prescribed zero pattern; see, for example, \cite{BrualdiRyser1991,GodsilRoyle2001,HornJohnson2013,LovaszPlummer1986}.  Related inertia questions for graph adjacency matrices have been studied extensively.  In particular, Ma, Yang, and Li \cite{MaYangLi2013} study positive and negative inertia indices of unweighted graph adjacency matrices for trees, unicyclic graphs, and bicyclic graphs, while Li and Song \cite{LiSong2013} treat weighted graph adjacency matrices in analogous classes.  For zero-diagonal weighted adjacency matrices of bipartite graphs, the symmetry of the spectrum about the origin implies that the positive and negative inertia indices coincide and are both equal to one half of the rank.  The form needed here is slightly different: we allow an arbitrary nonnegative diagonal perturbation of a bipartite-supported symmetric matrix.  This extension follows immediately from Weyl monotonicity, but it is essential for Markov chains, where diagonal entries represent holding probabilities (i.e.
 the probability that the chain remains at the same state for one transition step).  For reversible Markov chains, the usual symmetrization  transfers this estimate directly to the transition matrix (see Subsection 2.2 below).  For reducible tridiagonal stochastic matrices, Frobenius normal form reduces the problem to irreducible tridiagonal blocks and singleton blocks as we shall see below.

This paper is organized as follows. In Section 2, we collect some of the basic facts needed for our study. Section 3 contains our main results which can be  viewed as a constrained inverse inertia problem: among
graph-supported symmetric matrices with nonnegative diagonal, we bound the
negative index in terms of a purely combinatorial invariant, the matching
number of the support graph.  The proof combines three standard ingredients: monotonicity of negative inertia under positive semidefinite diagonal perturbations, the singular value structure of the bipartite block matrix $\bigl(\begin{smallmatrix}0&B\\ B^T&0\end{smallmatrix}\bigr)$, and the rank--term-rank bound for bipartite support patterns.  Applying this principle to reversible Markov chains gives a matching-number bound on the number of negative eigenvalues.  In the final section, we show that the tridiagonal bound is sharp in every dimension and we explore a possible extension.

\section{Preliminaries}

 In this section, we collect some of the basic facts needed for our analysis.

First,   if $M$ is a real matrix all of whose eigenvalues are real,  then the \emph{negative index}  of   $M$, denoted by
\(
        \numinus(M)
\)
 is the number of negative eigenvalues of $M$ counting algebraic multiplicity.
We shall use the term inertia in the standard sense for real symmetric matrices,
but our results concern only the negative index 
\(
        \nu_-(S).
\)
  With this notation,  a good starting point is the following lemma which is needed and will be used later.

\begin{lemma}
\label{lem:monotonicity}
Let \(A\) be a real symmetric \(n\times n\) matrix and let \(D\succeq0\)  (i.e. positive semi-definite).
Then
\[
        \numinus(A+D)\leq \numinus(A).
\]
\end{lemma}

\begin{proof}
We use Weyl's monotonicity theorem for Hermitian matrices (\cite{HornJohnson2013}). Since \(D\succeq0\), we have
\[
        (A+D)-A=D\succeq0.
\]
Thus, \(A\preceq A+D\) in the Loewner order; see, for example,
\cite{HornJohnson2013}.  Therefore, if the eigenvalues are arranged in
nondecreasing order,
\[
        \lambda_1(A)\leq\lambda_2(A)\leq\cdots\leq\lambda_n(A),
\]
and similarly for \(A+D\), Weyl's monotonicity theorem gives
\[
        \lambda_j(A)
        \leq
        \lambda_j(A+D),
        \qquad j=1,\ldots,n.
\]

Let
\[
        m=\numinus(A+D).
\]
Then, the first \(m\) eigenvalues of \(A+D\) are negative:
\[
        \lambda_1(A+D),\ldots,\lambda_m(A+D)<0.
\]
By the preceding eigenvalue inequalities, for every \(1\leq j\leq m\),
\[
        \lambda_j(A)
        \leq
        \lambda_j(A+D)
        <0.
\]
Thus,  \(A\) has at least \(m\) negative eigenvalues.  Hence
\[
        m\leq \numinus(A).
\]
Since \(m=\numinus(A+D)\), we conclude that
\[
        \numinus(A+D)\leq \numinus(A).
\]
\end{proof}

\subsection{Support graphs and matchings}

Let \(G=(V,E)\) be a finite simple graph with \(V=\{1,\ldots,n\}\), and whose adjacency matrix is $A_G$.  A real
symmetric matrix \(S=(s_{ij})\) is said to be \emph{supported on \(G\)} if
\[
        s_{ij}=0
        \qquad\text{whenever}\qquad
        i\neq j \ \text{and}\ \{i,j\}\notin E.
\]  Equivalently,
\[
        (J-I-A_G)\circ S=0,
\]
where \(J\) is the \emph{all-ones} matrix, $I$ is the identity matrix of appropriate size and \(\circ\) denotes \emph{the Hadamard product}.
This formulation emphasizes that support is a zero-pattern condition.
Thus, the graph \(G\) specifies the allowed off-diagonal nonzero positions of
\(S\).  The diagonal entries are unrestricted by this convention; in particular,
diagonal entries do not contribute edges to the support graph.

For a not necessarily symmetric \(n\times n\) matrix \(P=(p_{ij})\), the
\emph{off-diagonal support graph} of \(P\) is the simple undirected graph on
\(\{1,\ldots,n\}\) in which two distinct vertices \(i\) and \(j\) are adjacent
if and only if
\[
        p_{ij}\neq0
        \qquad\text{or}\qquad
        p_{ji}\neq0.
\]
Thus, this graph records whether there is an off-diagonal interaction between
\(i\) and \(j\), while forgetting its direction and ignoring its  diagonal entries \(p_{ii}\).
  In particular, the three cases
\[
        (p_{ij}>0,\,p_{ji}=0),\qquad
        (p_{ij}=0,\,p_{ji}>0),\qquad
        (p_{ij}>0,\,p_{ji}>0)
\]
all give the same undirected edge \(\{i,j\}\).

Equivalently, we shall say that the off-diagonal support graph of \(P\)  \emph{is contained}
in \(G=(V,E)\) means that, for every pair of distinct vertices \(i,j\),
\[
        \{i,j\}\notin E
        \quad\Longrightarrow\quad
        p_{ij}=0 \ \text{and}\ p_{ji}=0.
\]

In particular, if \(P\) is tridiagonal, then its off-diagonal support graph has
edges only between consecutive indices.  Hence, it is a subgraph of the path
\(
       \Path_n.
\)
Consequently, it is bipartite, with bipartition given by the odd and even
indices:
\[
        \{i: i \text{ is odd}\}
        \qquad\text{and}\qquad
        \{i: i \text{ is even}\}.
\]
The diagonal entries of \(P\) do not affect this conclusion.

\begin{example}\label{ex:support-graph}
Let
\[
P=
\begin{pmatrix}
\frac12 & \frac12 & 0 & 0\\
\frac13 & \frac13 & \frac13 & 0\\
0 & 0 & \frac12 & \frac12\\
0 & 0 & \frac14 & \frac34
\end{pmatrix}.
\]
The nonzero off-diagonal entries occur only in the positions corresponding to
the unordered pairs
\[
        \{1,2\},\qquad \{2,3\},\qquad \{3,4\}.
\]
Therefore,  the off-diagonal support graph of \(P\) is exactly the path graph
\(\Path_4\).  Notice that the edge \(\{2,3\}\) is present even though only one
direction occurs:
\[
        p_{23}>0,
        \qquad
        p_{32}=0.
\]
Thus, an undirected edge is included whenever at least one of \(p_{ij}\) and
\(p_{ji}\) is nonzero.  
\end{example}
A \emph{matching} in $G$ is a set of pairwise vertex-disjoint edges.  The \emph{matching number} of $G$ is denoted by $\mu(G)$, is the maximum number of vertex-disjoint edges.  For the path graph $\Path_n$, clearly one has
\[
        \mu(\Path_n)=\left\lfloor\frac n2\right\rfloor.
\]
Indeed, the alternating edges $\{1,2\},\{3,4\},\ldots$ give a matching of size $\lfloor n/2\rfloor$, while every matching uses two vertices per edge and therefore has size at most $\lfloor n/2\rfloor$.

A graph is \emph{bipartite}  if $V=X\sqcup Y$ and every edge has one endpoint in $X$ and the other in $Y$.
Given any \(X\times Y\) matrix \(B=(B_{xy})\), its \emph{bipartite support
graph} is the bipartite graph \(G_B\) with vertex set \(X\sqcup Y\), where the
elements of \(X\) index the rows of \(B\) and the elements of \(Y\) index the
columns.  For \(x\in X\) and \(y\in Y\), we place an edge between \(x\) and
\(y\) precisely when the corresponding matrix entry is nonzero; that is,
\[
        \{x,y\}\in E(G_B)
        \quad\Longleftrightarrow\quad
        B_{xy}\neq 0.
\]
Thus, \(G_B\) exists canonically for every \(X\times Y\) matrix \(B\), and it is
bipartite by construction, since every edge joins a row-vertex to a
column-vertex. Consequently, choosing nonzero entries of \(B\) with no two in the same row or column is exactly equivalent to choosing edges of \(G_B\) with no common endpoint.  
 At this point, we recall the definition of  \emph{termrank}$(B)$ as the maximum number of nonzero entries of $B$ with no two in the same row or column. Therefore,  in our notation, we have
\[
        \termrank(B)=\mu(G_B).
\]
 On the other hand, if  $\rank(B)=r$, then $B$ has a nonzero $r\times r$ minor.  Expanding this determinant, at least one permutation term is nonzero, and such  term selects $r$ nonzero entries of $B$ with no two in the same row or column.  Hence, these entries form a matching of size $r$ in $G_B$ so that the following holds
\begin{eqnarray}
        \rank(B)\leq\termrank(B)=\mu(G_B).
\end{eqnarray}
This is the standard rank--term-rank bound in combinatorial matrix theory; see, for example, \cite[Chapter 4]{BrualdiRyser1991}. We illustrate the above concepts with the  following simple example.                                                  
\begin{example}
Let
\(
        X=\{x_1,x_2\}, \
        Y=\{y_1,y_2,y_3\},
\)
and consider the \(X\times Y\) matrix
\(
B=
\begin{pmatrix}
1 & 0 & 2\\
0 & 3 & 0
\end{pmatrix}.
\)
The nonzero entries of \(B\) are
\(
        B_{x_1y_1}=1, \
        B_{x_1y_3}=2, \
        B_{x_2y_2}=3.
\)
Hence,  the bipartite support graph \(G_B\) has vertex set
\(
        V(G_B)=X\sqcup Y
\)
and edge set
\(
        E(G_B)=
        \bigl\{
        \{x_1,y_1\},
        \{x_1,y_3\},
        \{x_2,y_2\}
        \bigr\}.
\)
Therefore, each nonzero entry of \(B\) gives one edge joining a row-vertex in \(X\)
to a column-vertex in \(Y\).
  Choose
\(
        B_{x_1y_1} \ \text{and}  \
        B_{x_2y_2}
\)
as these two entries are nonzero and lie in different rows and different columns.
Thus,
\(
        \termrank(B)\geq 2.
\)
However, \(B\) has only two rows, so no selection can contain more
than two entries with distinct rows.  Hence,
\(
        \termrank(B)=2.
\) \\
On the other hand, in the bipartite support graph \(G_B\), the two edges
\(
        \{x_1,y_1\}
        \  \text{and} \
        \{x_2,y_2\}
\)
are disjoint, and therefore form a matching of size \(2\).  Since the graph has
only two vertices on the \(X\)-side, no matching can have more than two edges.
Thus,
\(
        \mu(G_B)=2
\)
and consequently,
\(
        \termrank(B)=\mu(G_B)=2.
\)
\end{example}
\subsection{Reversible Markov chains}

Let $P=(p_{ij})$ be an $n\times n$ stochastic matrix.
A positive \emph{probability vector}  (i.e. each of its entries is  $>0$ and sum of its components is 1)  $\pi=(\pi_i)$ is stationary if $\pi^TP=\pi^T$.  The chain is \emph{reversible with respect to $\pi$}  if
\[
        \pi_i p_{ij}=\pi_jp_{ji}
        \qquad\text{for all }i,j.
\]
If the $n\times n$ diagonal matrix $\Pi$ is given by  $\Pi=\diag(\pi_1,\ldots,\pi_n)$ , then reversibility implies that
\begin{eqnarray}
        S=\Pi^{1/2}P\Pi^{-1/2}
\end{eqnarray}
is real symmetric.  Indeed,
\[
        s_{ij}=\sqrt{\frac{\pi_i}{\pi_j}}p_{ij}
        =\sqrt{\frac{\pi_j}{\pi_i}}p_{ji}
        =s_{ji}.
\]
Moreover, $S$ is similar to $P$, hence $P$ and $S$ have the same eigenvalues. For more information on this topic see, for example, \cite{Norris1997}.
\section{Main results}
This section  deals with proving a matching-based bound for the negative
index in 3 cases: a symmetric matrix with nonnegative diagonal entries,  a reversible chain  and a tridiagonal stochastic matrix where each of them  is  supported on a bipartite graph.

\subsection{Symmetric matrix with nonnegative diagonal entries  supported on a bipartite graph}
\begin{theorem}
\label{thm:graph-inertia}
Let \(G=(V,E)\) be a finite bipartite graph with \(V=\{1,\ldots,n\}\), and let \(S=(s_{ij})\) be a real symmetric
matrix supported on \(G\).  Suppose that
\[
        s_{ii}\geq0
        \qquad
        (i=1,\ldots,n).
\]
Then,
\[
        \numinus(S)\leq \mu(G).
\]
Equivalently, \(S\) has at least \(n-\mu(G)\) nonnegative eigenvalues, counted
with multiplicity.
\end{theorem}

\begin{proof}
Write
\[
        S=D+W,
        \qquad
        D=\operatorname{diag}(s_{11},\ldots,s_{nn}),
\]
where \(W\) is the hollow part of \(S\), namely
\[
        w_{ij}=
        \begin{cases}
        s_{ij}, & i\neq j,\\
        0, & i=j.
        \end{cases}
\]
Since \(s_{ii}\geq0\) for every \(i\), then we have
\[
        D\succeq0.
\]
By Lemma~\ref{lem:monotonicity},
\[
        \numinus(S)
        =
        \numinus(W+D)
        \leq
        \numinus(W).
\]

Now let
\[
        V=X\sqcup Y
\]
be a bipartition of \(G\).  After a simultaneous permutation of rows and
columns, which does not change the inertia, we may assume that the vertices in
\(X\) precede those in \(Y\).  Since \(S\) is supported on the bipartite graph
\(G\), there are no off-diagonal entries of \(W\) joining two vertices both in
\(X\), nor two vertices both in \(Y\).  Hence \(W\) has the block form
\[
        W=
        \begin{pmatrix}
        0 & B\\
        B^T & 0
        \end{pmatrix},
\]
where \(B\) is the submatrix with rows indexed by \(X\) and columns indexed by
\(Y\).

If \(r=\rank(B)\), then we denote by
\[
        \sigma_1,\ldots,\sigma_r,
\]
to be the nonzero singular values of \(B\) but then  the nonzero eigenvalues of \(W\)  are  given by \[
        \sigma_1,\ldots,\sigma_r,
        -\sigma_1,\ldots,-\sigma_r,
\]  so that 
\[
        \numinus(W)=r=\rank(B).
\]

Let \(G_B\) be the bipartite support graph of \(B\), with vertex classes
\(X\) and \(Y\), and with an edge \(xy\) whenever \(B_{xy}\neq0\).  Since
\(S\) is supported on \(G\), every edge of \(G_B\) is an edge of \(G\).  Hence, 
\(G_B\) is a subgraph of \(G\).  Therefore,  every matching in \(G_B\) is also a
matching in \(G\), and so
\[
        \mu(G_B)\leq \mu(G).
\]


Combining these inequalities and in view of (2.1), we obtain
\[
        \numinus(S)
        \leq
        \numinus(W)
        =
        \rank(B)
        \leq
        \termrank(B)
        =
        \mu(G_B)
        \leq
        \mu(G).
\]
This proves
\(
        \numinus(S)\leq \mu(G),
\)
and the proof now can be easily completed.
\end{proof}
\begin{remark}\label{rem:support-pattern}
Theorem~\ref{thm:graph-inertia} is purely a support-pattern statement for real symmetric matrices.  No stochasticity, row-sum condition, or nonnegativity of the off-diagonal entries is used.  The matching number enters only through the identity
\[
        \termrank(B)=\mu(G_B).
\]
 Stochasticity enters later only when a transition matrix is related to a symmetric similar matrix and when the diagonal entries $p_{ii}\geq0$ are interpreted as holding probabilities.
\end{remark}

\subsection{Reversible chains on bipartite graphs}
We start with the following theorem that deals with the case of reversible chains on bipartite graphs.
\begin{theorem}\label{thm:reversible-bipartite}
Let \(P=(p_{ij})_{i,j=1}^n\) be a finite Markov transition matrix that is
reversible with respect to a positive stationary probability vector
\(\pi=(\pi_1,\ldots,\pi_n)\). Suppose that the off-diagonal support graph of \(P\)
is contained in a bipartite graph \(G\). Then, all eigenvalues of \(P\) are real
and
\[
        \numinus(P)\leq \mu(G).
\]
Consequently, \(P\) has at least \(n-\mu(G)\) nonnegative eigenvalues, counted
with their multiplicities.
\end{theorem}
\begin{proof}
Let $\Pi=\diag(\pi_1,\ldots,\pi_n)$ and in view of (2.2), we have
\[
        S=\Pi^{1/2}P\Pi^{-1/2}.
\]
By reversibility, $S$ is real symmetric.  Since $S$ is similar to $P$, the two matrices have the same eigenvalues.  For $i\neq j$,
\[
        s_{ij}=0\quad\Longleftrightarrow\quad p_{ij}=0,
\]
because $\pi_i>0$ for every $i$.  Thus, the off-diagonal support of $S$ is contained in $G$.  Moreover,
\[
        s_{ii}=p_{ii}\geq0.
\]
By Theorem~\ref{thm:graph-inertia},
\[
        \numinus(S)\leq\mu(G).
\]
Since $S$ is similar to $P$, this gives
\[
        \numinus(P)=\numinus(S)\leq\mu(G).
\]
\end{proof}

\subsection{Tridiagonal stochastic matrices}

We now prove the main result here which deliberately combines the birth-death-chain and tridiagonal-matrix formulations, since they describe the same class of stochastic matrices.  Reducible matrices are handled by a direct Frobenius-block reduction.

\begin{lemma}
\label{lem:irreducible-tridiagonal-symmetrizable}
Let $A=(a_{ij})$ be an $n\times n$ irreducible nonnegative tridiagonal matrix.  Then, $A$ is diagonally symmetrizable.  More precisely, there exists a positive diagonal matrix $D$ such that $DAD^{-1}$ is real symmetric.
\end{lemma}

\begin{proof}
Irreducibility means that every element on the subdiagonal as well as on the superdiagonal must be nonzero.  Hence, for each $i=1,\ldots,n-1$, both $a_{i,i+1}$ and $a_{i+1,i}$ are positive.  Choose $\rho_1=1$, and define recursively
\[
        \rho_{i+1}=\rho_i\frac{a_{i,i+1}}{a_{i+1,i}},
        \qquad i=1,\ldots,n-1.
\]
Then, $\rho_i>0$ and
\[
        \rho_i a_{i,i+1}=\rho_{i+1}a_{i+1,i}.
\]
Let
\(
        D=\diag(\sqrt{\rho_1},\ldots,\sqrt{\rho_n}).
\)
Then, we can write
\[
        (DAD^{-1})_{i,i+1}
        =
        \sqrt{\frac{\rho_i}{\rho_{i+1}}}\,a_{i,i+1}
        =
        \sqrt{\frac{\rho_{i+1}}{\rho_i}}\,a_{i+1,i}
        =
        (DAD^{-1})_{i+1,i}.
\]
All other off-diagonal entries are zero except the corresponding tridiagonal ones, and the diagonal entries are unchanged.  Thus, $DAD^{-1}$ is real symmetric.
\end{proof}

 The next proposition is needed for the proof of our main results. For its proof, we use a graph-theoretic approach and it would be very helpful for the reader to check the steps  of this  proof on  one of the next two examples that proceed it.
\begin{proposition}\label{prop:tridiagonal-reduction}
Let \(P\) be an \(n\times n\) tridiagonal stochastic matrix. Then \(P\) is
permutation-similar to a block upper triangular matrix, that is, there exists
a permutation matrix \(Q\) such that
\[
Q^TPQ=
\begin{pmatrix}
P_1 & * & \cdots & *\\
0 & P_2 & \ddots & \vdots\\
\vdots & \ddots & \ddots & *\\
0 & \cdots & 0 & P_m
\end{pmatrix},
\]
where each diagonal block \(P_\alpha\) is either \(1\times1\), or is an
irreducible nonnegative tridiagonal substochastic matrix. 
 Finally, the eigenvalues of \(P\) are the union, counted with algebraic
multiplicity, of the eigenvalues of the diagonal blocks \(P_\alpha\).
\end{proposition}
\begin{proof}
Associate to \(P=(p_{ij})\) the directed graph \(\Gamma(P)\) with vertex set
\(\{1,\ldots,n\}\), where there is a directed edge \(i\to j\) precisely when
\[
        p_{ij}>0.
\]
Since \(P\) is tridiagonal, we have \(p_{ij}=0\) whenever \(|i-j|>1\). Hence
every directed edge \(i\to j\) in \(\Gamma(P)\) satisfies
\[
        |i-j|\leq 1.
\]
Thus,  every non-loop edge joins two consecutive indices.

Let
\[
        C_1,\ldots,C_m
\]
be the strongly connected components of \(\Gamma(P)\). We first prove that each
\(C_\alpha\) is an interval of \(\{1,\ldots,n\}\), meaning a set of the form
\[
        \{a,a+1,\ldots,b\}
\]
for some \(1\leq a\leq b\leq n\).

Fix one strongly connected component \(C\). Suppose \(i,k\in C\) with \(i<k\). If \(k=i+1\), then there is no integer
strictly between \(i\) and \(k\), so there is nothing to prove. Otherwise,
\(k\geq i+2\), and we may choose an integer \(j\) such that
\[
        i<j<k.
\]
We shall show that \(j\in C\).
Since \(i\) and \(k\) lie in the same strongly connected component, there is a
directed path from \(i\) to \(k\):
\[
        i=i_0\to i_1\to \cdots \to i_r=k.
\]
For each \(\ell=1,\ldots,r\), the edge \(i_{\ell-1}\to i_\ell\) is an edge of
\(\Gamma(P)\), and therefore
\[
        |i_\ell-i_{\ell-1}|\leq 1.
\]
The integer sequence \(i_0,i_1,\ldots,i_r\) starts at \(i<j\), ends at
\(k>j\), and changes by at most one at each step. Hence it must pass through
\(j\). Therefore,  there exists \(t\in\{1,\ldots,r-1\}\) such that
\[
        i_t=j.
\]
It follows that \(i\) reaches \(j\), and \(j\) reaches \(k\).

Similarly, since \(C\) is strongly connected, there is a directed path from
\(k\) back to \(i\). The same argument shows that this path must also pass
through \(j\). Hence \(k\) reaches \(j\), and \(j\) reaches \(i\).

Therefore, aches \(j\), and \(j\) reaches \(i\). Therefore \(j\) lies in the
same strongly connected component as \(i\). Since \(i\in C\), we conclude that
\(j\in C\). Hence every integer lying between two elements of \(C\) also belongs
to \(C\). Thus, \(C\) is an interval.


Now form \emph{the condensation graph}  of \(\Gamma(P)\). It is the directed graph whose vertices are the strongly connected components \(C_1,\ldots,C_m\), and where there is an edge
\(C_\alpha\to C_\beta\) if some vertex in \(C_\alpha\) has a directed edge to some vertex in \(C_\beta\). This condensation graph
is acyclic. Indeed, if there
were a directed cycle
\[
        C_{\alpha_1}\to C_{\alpha_2}\to \cdots \to C_{\alpha_r}
        \to C_{\alpha_1},
\]
then every component on this cycle would be reachable from every other
component on the cycle. Hence the union
\[
        C_{\alpha_1}\cup C_{\alpha_2}\cup \cdots \cup C_{\alpha_r}
\]
would be strongly connected, contradicting the maximality of the strongly
connected components.

Now, since the condensation graph is finite and has no directed cycles,
we may relabel the components \(C_1,\ldots,C_m\)  as 
\[
        E_1,\ldots,E_m
\]
so that every edge between distinct components goes from a lower-indexed
component to a higher-indexed component; that is, whenever
\[
        E_\alpha\to E_\beta
        \qquad\text{and}\qquad \alpha\neq\beta,
\]
we have
\[
        \alpha<\beta.
\]

If two components
are not joined by an edge in the condensation graph, their relative order may
be chosen arbitrarily.

It should be noted here that the ordering of the components is not necessarily unique.   Once such an admissible ordering
\[
        E_1,\ldots,E_m
\]
has been fixed, we define \(Q\) to  be the permutation matrix corresponding to this ordering of the vertices in the following way. 
  List the vertices of \(E_1\) first, then those of \(E_2\), and so on, always using
the natural increasing order inside each \(E_\alpha\), which we have shown to be an interval.
  If the resulting
ordered list is
\[
        \sigma(1),\ldots,\sigma(n),
\]
then
\(
        Q=\bigl[e_{\sigma(1)}\ e_{\sigma(2)}\ \cdots\ e_{\sigma(n)}\bigr], 
\) where $e_i$ is  the ith column of the $n\times n$ identity matrix.
Different admissible orderings may give different permutation matrices \(Q\),
but each such \(Q\) gives a block upper triangular matrix \(Q^TPQ\).

With this choice of \(Q\), the matrix \(Q^TPQ\) is block upper triangular.  To
see this, suppose that a block below the diagonal were nonzero.  Then, there
would exist indices \(i\in E_\beta\) and \(j\in E_\alpha\), with
\(\beta>\alpha\), such that
\[
        p_{ij}>0.
\]
This would give an edge \(i\to j\) in \(\Gamma(P)\), and hence an edge
\[
        E_\beta\to E_\alpha
\]
in the condensation graph.  But this contradicts the chosen ordering, since
\(\beta>\alpha\).  Therefore,  all blocks below the diagonal are zero, and hence we can write
\[
Q^TPQ=
\begin{pmatrix}
P_1 & * & \cdots & *\\
0 & P_2 & \ddots & \vdots\\
\vdots & \ddots & \ddots & *\\
0 & \cdots & 0 & P_m
\end{pmatrix}.
\]

We now describe the diagonal blocks. Fix a component \(E_\alpha\). Since each \(E_\alpha\) is an interval, we may write
\[
        E_\alpha=\{a_\alpha,a_\alpha+1,\ldots,b_\alpha\}.
\]
Because the indices inside \(E_\alpha\) are listed in increasing order, the
corresponding diagonal block \(P_\alpha\) is precisely the principal submatrix
of \(P\) indexed by \(\{a_\alpha,a_\alpha+1,\ldots,b_\alpha\}\). Since \(P\) is tridiagonal, then
the same condition holds inside the principal submatrix \(P_\alpha\) so that
\(P_\alpha\) is tridiagonal.

If \(E_\alpha\) contains more than one vertex, then the induced directed graph
on \(E_\alpha\) is strongly connected. To see this, take \(u,v\in E_\alpha\).
Since \(u\) and \(v\) are in the same strongly connected component of
\(\Gamma(P)\), there is a directed path from \(u\) to \(v\). This path cannot
leave \(E_\alpha\). Indeed, if it passed through some vertex \(w\notin
E_\alpha\), then \(u\) would reach \(w\), and \(w\) would reach \(v\). Since
\(v\) reaches \(u\), it would follow that \(w\) and \(u\) are mutually
reachable, so \(w\) would belong to \(E_\alpha\), a contradiction. Hence every
path between two vertices of \(E_\alpha\) stays inside \(E_\alpha\). Therefore
the directed graph of the block \(P_\alpha\) is strongly connected, and so
\(P_\alpha\) is irreducible as a nonnegative matrix.

Since \(P\) is stochastic, then so is the matrix $Q^TPQ$. For \(i\in
E_\alpha\), the corresponding row sum inside the block \(P_\alpha\) is
\[
        \sum_{j\in E_\alpha} p_{ij}.
\]
Because all entries of \(P\) are nonnegative,
\[
        \sum_{j\in E_\alpha} p_{ij}
        \leq
        \sum_{j=1}^n p_{ij}
        =
        1.
\]
Thus,  each diagonal block \(P_\alpha\) is substochastic. It is stochastic
precisely when
\[
        \sum_{j\in E_\alpha} p_{ij}=1
        \qquad\text{for every } i\in E_\alpha,
\]
equivalently, when
\[
        p_{ij}=0
        \qquad\text{whenever } i\in E_\alpha
        \text{ and } j\notin E_\alpha.
\]

Finally,  the proof can be easily completed by noticing that   \(Q^TPQ\) is block upper triangular which is similar to $P$.   
\end{proof}

\begin{example}
Consider the tridiagonal stochastic matrix
\[
P=
\begin{pmatrix}
\frac12 & \frac12 & 0        & 0        & 0\\
\frac12 & \frac12 & 0        & 0        & 0\\
0       & \frac25 & \frac15 & \frac25 & 0\\
0       & 0       & 0        & \frac12 & \frac12\\
0       & 0       & 0        & \frac12 & \frac12
\end{pmatrix}.
\]

Associate to \(P\) the directed graph \(\Gamma(P)\) on the vertex set
\(\{1,2,3,4,5\}\), where \(i\to j\) whenever \(p_{ij}>0\).  Ignoring loops,
the nontrivial directed edges are
\[
        1\leftrightarrow 2,
        \qquad
        3\to 2,
        \qquad
        3\to 4,
        \qquad
        4\leftrightarrow 5.
\]

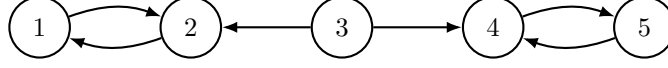
\begin{figure}[ht]
\centering
\begin{tikzpicture}[
    vertex/.style={circle, draw, thick, minimum size=8mm},
    edge/.style={-{Latex[length=2mm]}, thick},
    bend angle=20
]

\node[vertex] (1) at (0,0) {$1$};
\node[vertex] (2) at (2,0) {$2$};
\node[vertex] (3) at (4,0) {$3$};
\node[vertex] (4) at (6,0) {$4$};
\node[vertex] (5) at (8,0) {$5$};

\draw[edge, bend left] (1) to (2);
\draw[edge, bend left] (2) to (1);

\draw[edge] (3) to (2);
\draw[edge] (3) to (4);

\draw[edge, bend left] (4) to (5);
\draw[edge, bend left] (5) to (4);


\end{tikzpicture}
\caption{The directed graph \(\Gamma(P)\) associated with \(P\). There is an edge
\(i\to j\) whenever \(p_{ij}>0\).}
\label{fig:graph-P}
\end{figure}

Thus the strongly connected components are
\[
        C_1=\{1,2\},
        \qquad
        C_2=\{3\},
        \qquad
        C_3=\{4,5\}.
\]
Each of these sets is an interval of \(\{1,\ldots,5\}\).

The condensation graph has an edge from \(\{3\}\) to \(\{1,2\}\), because
\(3\to 2\), and an edge from \(\{3\}\) to \(\{4,5\}\), because \(3\to 4\).
Hence we relabel the same components as
\[
        E_1=\{3\},
        \qquad
        E_2=\{1,2\},
        \qquad
        E_3=\{4,5\}.
\]

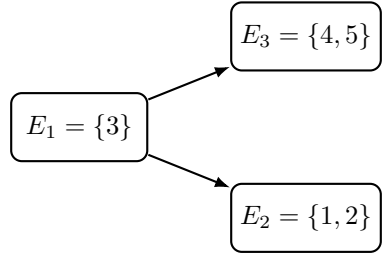
\begin{figure}[ht]
\centering
\begin{tikzpicture}[
    component/.style={rectangle, draw, thick, rounded corners, minimum width=18mm, minimum height=9mm},
    edge/.style={-{Latex[length=2mm]}, thick}
]

\node[component] (E1) at (0,0) {$E_1=\{3\}$};
\node[component] (E2) at (3,-1.2) {$E_2=\{1,2\}$};
\node[component] (E3) at (3,1.2) {$E_3=\{4,5\}$};

\draw[edge] (E1) to (E2);
\draw[edge] (E1) to (E3);

\end{tikzpicture}
\caption{The condensation graph of \(\Gamma(P)\). Each vertex represents a
strongly connected component of \(\Gamma(P)\).}
\label{fig:condensation-graph}
\end{figure}

With this ordering, every edge between distinct components goes from a
lower-indexed component to a higher-indexed component.

We now list the vertices by first listing the elements of \(E_1\), then those
of \(E_2\), and then those of \(E_3\).  Since each \(E_\alpha\) is an interval,
we list the elements inside each \(E_\alpha\) in increasing order.  Thus the
new ordering of the vertices is
\[
        3,\ 1,\ 2,\ 4,\ 5.
\]
Let \(Q\) be the permutation matrix corresponding to this ordering, namely
\[
Q=
\begin{pmatrix}
0&1&0&0&0\\
0&0&1&0&0\\
1&0&0&0&0\\
0&0&0&1&0\\
0&0&0&0&1
\end{pmatrix}.
\]
Equivalently, the columns of \(Q\) are
\[
        e_3,\ e_1,\ e_2,\ e_4,\ e_5.
\]
A direct computation gives
\[
Q^TPQ
=
\begin{pmatrix}
\frac15 & 0       & \frac25 & \frac25 & 0\\
0       & \frac12 & \frac12 & 0       & 0\\
0       & \frac12 & \frac12 & 0       & 0\\
0       & 0       & 0       & \frac12 & \frac12\\
0       & 0       & 0       & \frac12 & \frac12
\end{pmatrix}.
\]

\end{example}

\begin{example}
Consider the \(10\times 10\) tridiagonal stochastic matrix
\[
P=
\begin{pmatrix}
\frac12 & \frac12 & 0        & 0        & 0        & 0        & 0        & 0        & 0        & 0\\
\frac14 & \frac12 & \frac14 & 0        & 0        & 0        & 0        & 0        & 0        & 0\\
0       & \frac12 & \frac12 & 0        & 0        & 0        & 0        & 0        & 0        & 0\\
0       & 0       & \frac14 & \frac12 & \frac14 & 0        & 0        & 0        & 0        & 0\\
0       & 0       & 0        & 0        & \frac12 & \frac12 & 0        & 0        & 0        & 0\\
0       & 0       & 0        & 0        & 0        & \frac12 & \frac12 & 0        & 0        & 0\\
0       & 0       & 0        & 0        & 0        & \frac14 & \frac12 & \frac14 & 0        & 0\\
0       & 0       & 0        & 0        & 0        & 0        & 0        & \frac12 & \frac12 & 0\\
0       & 0       & 0        & 0        & 0        & 0        & 0        & \frac14 & \frac12 & \frac14\\
0       & 0       & 0        & 0        & 0        & 0        & 0        & 0        & \frac12 & \frac12
\end{pmatrix}.
\]

Associate to \(P\) the directed graph \(\Gamma(P)\) on the vertex set
\(\{1,\ldots,10\}\), where \(i\to j\) whenever \(p_{ij}>0\).  We omit loops in
the figure for clarity.  The non-loop directed edges are
\[
1\leftrightarrow 2,\qquad
2\leftrightarrow 3,\qquad
4\to 3,\qquad
4\to 5,\qquad
5\to 6,
\]
\[
6\leftrightarrow 7,\qquad
7\to 8,\qquad
8\leftrightarrow 9,\qquad
9\leftrightarrow 10.
\]

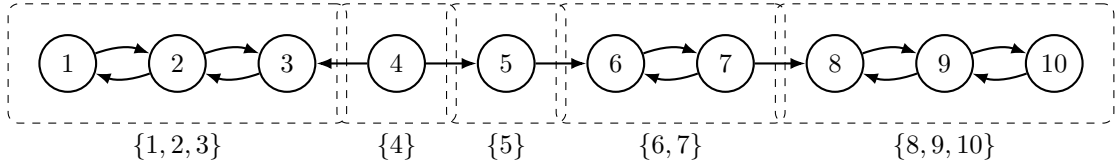
\begin{figure}[ht]
\centering
\begin{tikzpicture}[
    vertex/.style={circle, draw, thick, minimum size=7.5mm},
    edge/.style={-{Latex[length=2mm]}, thick},
    compbox/.style={draw, dashed, rounded corners, inner sep=4mm},
    bend angle=20
]

\node[vertex] (1) at (0,0) {$1$};
\node[vertex] (2) at (1.45,0) {$2$};
\node[vertex] (3) at (2.90,0) {$3$};
\node[vertex] (4) at (4.35,0) {$4$};
\node[vertex] (5) at (5.80,0) {$5$};
\node[vertex] (6) at (7.25,0) {$6$};
\node[vertex] (7) at (8.70,0) {$7$};
\node[vertex] (8) at (10.15,0) {$8$};
\node[vertex] (9) at (11.60,0) {$9$};
\node[vertex] (10) at (13.05,0) {$10$};

\draw[edge, bend left] (1) to (2);
\draw[edge, bend left] (2) to (1);

\draw[edge, bend left] (2) to (3);
\draw[edge, bend left] (3) to (2);

\draw[edge] (4) to (3);
\draw[edge] (4) to (5);

\draw[edge] (5) to (6);

\draw[edge, bend left] (6) to (7);
\draw[edge, bend left] (7) to (6);

\draw[edge] (7) to (8);

\draw[edge, bend left] (8) to (9);
\draw[edge, bend left] (9) to (8);

\draw[edge, bend left] (9) to (10);
\draw[edge, bend left] (10) to (9);

\node[compbox, fit={(1) (2) (3)}, label=below:{$\{1,2,3\}$}] {};
\node[compbox, fit={(4)}, label=below:{$\{4\}$}] {};
\node[compbox, fit={(5)}, label=below:{$\{5\}$}] {};
\node[compbox, fit={(6) (7)}, label=below:{$\{6,7\}$}] {};
\node[compbox, fit={(8) (9) (10)}, label=below:{$\{8,9,10\}$}] {};

\end{tikzpicture}
\caption{The directed graph \(\Gamma(P)\), with loops omitted. The dashed boxes
indicate the strongly connected components.}
\label{fig:graph-P-ten}
\end{figure}

From Figure~\ref{fig:graph-P-ten}, the strongly connected components are
\[
\{1,2,3\},\qquad
\{4\},\qquad
\{5\},\qquad
\{6,7\},\qquad
\{8,9,10\}.
\]
Each of these is an interval of \(\{1,\ldots,10\}\).

The condensation graph has one vertex for each strongly connected component.
There is an arrow between two components whenever there is a directed edge in
\(\Gamma(P)\) from a vertex in the first component to a vertex in the second.
Thus, the arrows are
\[
\{4\}\to \{1,2,3\},\qquad
\{4\}\to \{5\},\qquad
\{5\}\to \{6,7\},\qquad
\{6,7\}\to \{8,9,10\}.
\]

We therefore relabel the components  as
\[
E_1=\{4\},\qquad
E_2=\{1,2,3\},\qquad
E_3=\{5\},\qquad
E_4=\{6,7\},\qquad
E_5=\{8,9,10\}.
\]
With this ordering, every arrow in the condensation graph goes from a
lower-indexed component to a higher-indexed component.

\begin{figure}[ht]
\centering
\begin{tikzpicture}[
    component/.style={
        rectangle,
        draw,
        thick,
        rounded corners,
        minimum width=28mm,
        minimum height=10mm,
        align=center
    },
    edge/.style={-{Latex[length=2mm]}, thick}
]

\node[component] (E1) at (0,0) {$E_1=\{4\}$};
\node[component] (E2) at (4,1.35) {$E_2=\{1,2,3\}$};
\node[component] (E3) at (4,-1.35) {$E_3=\{5\}$};
\node[component] (E4) at (8,-1.35) {$E_4=\{6,7\}$};
\node[component] (E5) at (12,-1.35) {$E_5=\{8,9,10\}$};

\draw[edge] (E1) to (E2);
\draw[edge] (E1) to (E3);
\draw[edge] (E3) to (E4);
\draw[edge] (E4) to (E5);

\end{tikzpicture}
\caption{The condensation graph of \(\Gamma(P)\), written in Frobenius order.
Every arrow goes from a lower-indexed component to a higher-indexed component.}
\label{fig:condensation-P-ten}
\end{figure}
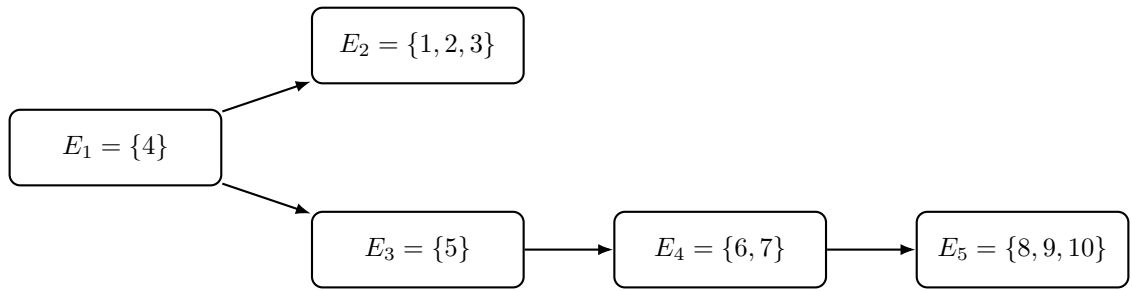

We now list the vertices by first listing the elements of \(E_1\), then those
of \(E_2\), and so on.  Inside each \(E_\alpha\), we list the vertices in their
natural increasing order.  Thus, the new order of the vertices is
\[
4,\ 1,\ 2,\ 3,\ 5,\ 6,\ 7,\ 8,\ 9,\ 10.
\]
Let \(Q\) be the corresponding permutation matrix. Equivalently, the columns of
\(Q\) are
\[
e_4,\ e_1,\ e_2,\ e_3,\ e_5,\ e_6,\ e_7,\ e_8,\ e_9,\ e_{10}.
\]
Then,  \(Q^TPQ\) is the matrix \(P\) written in this new order. Finally, a direct
calculation gives
\[
Q^TPQ=
\begin{pmatrix}
\frac12 & 0        & 0        & \frac14 & \frac14 & 0        & 0        & 0        & 0        & 0\\
0       & \frac12 & \frac12 & 0        & 0        & 0        & 0        & 0        & 0        & 0\\
0       & \frac14 & \frac12 & \frac14 & 0        & 0        & 0        & 0        & 0        & 0\\
0       & 0        & \frac12 & \frac12 & 0        & 0        & 0        & 0        & 0        & 0\\
0       & 0        & 0        & 0        & \frac12 & \frac12 & 0        & 0        & 0        & 0\\
0       & 0        & 0        & 0        & 0        & \frac12 & \frac12 & 0        & 0        & 0\\
0       & 0        & 0        & 0        & 0        & \frac14 & \frac12 & \frac14 & 0        & 0\\
0       & 0        & 0        & 0        & 0        & 0        & 0        & \frac12 & \frac12 & 0\\
0       & 0        & 0        & 0        & 0        & 0        & 0        & \frac14 & \frac12 & \frac14\\
0       & 0        & 0        & 0        & 0        & 0        & 0        & 0        & \frac12 & \frac12
\end{pmatrix}.
\]
\end{example}

Now we are in a position to prove one of the main results of this paper.
\begin{theorem}
\label{thm:birth-death}
Let $P$ be an $n\times n$ tridiagonal stochastic matrix.  Equivalently, let $P$ be the transition matrix of a finite birth-death chain on $\{1,\ldots,n\}$.  Then, all eigenvalues of $P$ are real and
\[
        \numinus(P)\leq\left\lfloor\frac n2\right\rfloor.
\]
Consequently,
\(
        \lambda_{\lceil n/2\rceil}(P)\geq0,
\)
 and in particular,
\(
        \lambda_2(P)\geq0 \ \text{ for } n\geq3.
\)
\end{theorem}

\begin{proof}
By Proposition~\ref{prop:tridiagonal-reduction}, $P$ is permutation similar to a block upper triangular matrix i.e. there exists a permutation matrix $Q$ such that
\[
Q^TP Q=
\begin{pmatrix}
P_1 & * & \cdots & *\\
0 & P_2 & \ddots & \vdots\\
\vdots & \ddots & \ddots & *\\
0 & \cdots & 0 & P_m
\end{pmatrix},
\]
where each diagonal block $P_\alpha$ is either $1\times1$, or is an irreducible nonnegative tridiagonal block.  The eigenvalues of $P$ are the union, counted with algebraic multiplicity, of the eigenvalues of the diagonal blocks $P_\alpha$.

Let $n_\alpha$ be the size of $P_\alpha$.  If $n_\alpha=1$, then $P_\alpha$ contributes its single diagonal entry as an eigenvalue.  This entry is nonnegative because $P\geq0$.  Hence
\[
        \numinus(P_\alpha)=0=\left\lfloor\frac{n_\alpha}{2}\right\rfloor.
\]

Now suppose $n_\alpha\geq2$.  Then $P_\alpha$ is an irreducible nonnegative tridiagonal block.  By Lemma~\ref{lem:irreducible-tridiagonal-symmetrizable}, there exists a positive diagonal matrix $D_\alpha$ such that
\[
        S_\alpha=D_\alpha P_\alpha D_\alpha^{-1}
\]
is real symmetric tridiagonal.  The diagonal entries of $S_\alpha$ are the same as those of $P_\alpha$, and are therefore nonnegative.  Since $S_\alpha$ is supported on the path graph $\Path_{n_\alpha}$, Theorem~\ref{thm:graph-inertia} gives
\[
        \numinus(P_\alpha)
        =
        \numinus(S_\alpha)
        \leq
        \mu(\Path_{n_\alpha})
        =
        \left\lfloor\frac{n_\alpha}{2}\right\rfloor.
\]

Summing over the diagonal blocks, we obtain
\[
        \numinus(P)
        =
        \sum_{\alpha=1}^m\numinus(P_\alpha)
        \leq
        \sum_{\alpha=1}^m\left\lfloor\frac{n_\alpha}{2}\right\rfloor
        \leq
        \left\lfloor
        \frac12\sum_{\alpha=1}^m n_\alpha
        \right\rfloor
        =
        \left\lfloor\frac n2\right\rfloor.
\]

Clearly, the number of nonnegative eigenvalues is given by
\[
        n-\left\lfloor\frac n2\right\rfloor
        =
        \left\lceil\frac n2\right\rceil.
\]
Since all eigenvalues are real, we order them as
\(
        \lambda_1(P)\geq\lambda_2(P)\geq\cdots\geq\lambda_n(P)
\)
so that
\(
        \lambda_{\lceil n/2\rceil}(P)\geq0.
\)
Finally, for $n\geq3$, we have $\lceil n/2\rceil\geq2$, and therefore,
\[
        \lambda_2(P)
        \geq
        \lambda_{\lceil n/2\rceil}(P)
        \geq0.
\]
\end{proof}


\section{ Sharpness,  Necessity of bipartiteness, Possible extension}
\subsection{Sharpness in every dimension}
The estimate in Theorem~\ref{thm:birth-death} is sharp for every \(n\geq2\).  We
illustrate this using the reflecting simple random walk on the path graph
\(\Path_n\) as shown in the following example.

\begin{example}  For \(n=2\), this transition matrix is
\[
        R_2=
        \begin{pmatrix}
        0 & 1\\
        1 & 0
        \end{pmatrix}.
\]
For \(n\geq3\), it is given by
\[
R_n=
\begin{pmatrix}
0 & 1 & 0 & \cdots & 0\\
\frac12 & 0 & \frac12 & \ddots & \vdots\\
0 & \frac12 & 0 & \ddots & 0\\
\vdots & \ddots & \ddots & \ddots & \frac12\\
0 & \cdots & 0 & 1 & 0
\end{pmatrix}.
\]
Equivalently,
\[
        (R_n)_{1,2}=1,
        \qquad
        (R_n)_{n,n-1}=1,
\]
and
\[
        (R_n)_{i,i-1}=(R_n)_{i,i+1}=\frac12
        \qquad (2\leq i\leq n-1).
\]

We claim that the eigenvalues of \(R_n\) are
\[
        \lambda_k=
        \cos\left(\frac{k\pi}{n-1}\right),
        \qquad
        k=0,1,\ldots,n-1.
\]
Set
\[
        \theta_k=\frac{k\pi}{n-1}.
\]
For each \(k=0,\ldots,n-1\), define the vector
\[
        f_k=
        \bigl(f_k(1),\ldots,f_k(n)\bigr)^T\in\mathbb{R}^n
\]
by
\[
        f_k(i)=\cos\bigl((i-1)\theta_k\bigr),
        \qquad i=1,\ldots,n.
\]
We verify that \(f_k\) is an eigenvector of \(R_n\) with eigenvalue
\[
        \lambda_k=\cos(\theta_k).
\]

For an interior index \(2\leq i\leq n-1\), we have
\[
        (R_nf_k)(i)
        =
        \frac12 f_k(i-1)+\frac12 f_k(i+1).
\]
Using the identity
\[
        \cos(x-\theta)+\cos(x+\theta)=2\cos(\theta)\cos(x),
\]
with \(x=(i-1)\theta_k\) and \(\theta=\theta_k\), we obtain
\[
        \frac12 f_k(i-1)+\frac12 f_k(i+1)
        =
        \cos(\theta_k)f_k(i).
\]
Hence,
\[
        (R_nf_k)(i)=\lambda_k f_k(i)
        \qquad (2\leq i\leq n-1).
\]

At the left endpoint,
\[
        (R_nf_k)(1)=f_k(2)=\cos(\theta_k).
\]
Since
\[
        f_k(1)=\cos(0)=1,
\]
we get
\[
        (R_nf_k)(1)
        =
        \cos(\theta_k)f_k(1)
        =
        \lambda_k f_k(1).
\]

At the right endpoint,
\[
        (R_nf_k)(n)=f_k(n-1).
\]
Now
\[
        f_k(n)=\cos((n-1)\theta_k)=\cos(k\pi),
\]
and
\[
        f_k(n-1)=\cos((n-2)\theta_k)=\cos(k\pi-\theta_k).
\]
Since \(k\) is an integer,
\[
        \sin(k\pi)=0.
\]
Therefore,
\[
        \cos(k\pi-\theta_k)
        =
        \cos(k\pi)\cos(\theta_k)+\sin(k\pi)\sin(\theta_k)
        =
        \cos(k\pi)\cos(\theta_k).
\]
Thus,
\[
        f_k(n-1)
        =
        \cos(\theta_k) f_k(n),
\]
and hence
\[
        (R_nf_k)(n)=\lambda_k f_k(n).
\]
Therefore,
\[
        R_nf_k=\lambda_k f_k.
\]

It remains to see that this list gives the full spectrum.  The numbers
\[
        \theta_k=\frac{k\pi}{n-1},
        \qquad k=0,1,\ldots,n-1,
\]
are distinct and lie in the interval \([0,\pi]\).  Since the cosine function is
strictly decreasing on \([0,\pi]\), the eigenvalues
\[
        \lambda_k=\cos(\theta_k),
        \qquad k=0,1,\ldots,n-1,
\]
are distinct.  Thus, we have found \(n\) distinct eigenvalues of the
\(n\times n\) matrix \(R_n\).  Hence they exhaust the spectrum of \(R_n\).

We now count the negative eigenvalues.  Since
\[
        0\leq \theta_k\leq \pi,
\]
we have
\[
        \cos(\theta_k)<0
        \quad\Longleftrightarrow\quad
        \theta_k>\frac{\pi}{2}.
\]
Equivalently,
\[
        \frac{k\pi}{n-1}>\frac{\pi}{2},
\]
or
\[
        k>\frac{n-1}{2}.
\]

We count the integers \(k\in\{0,1,\ldots,n-1\}\) satisfying this inequality.
If \(n\) is even, then \((n-1)/2\) is not an integer, and the condition is
equivalent to
\[
        k\geq \frac n2.
\]
Hence, the admissible integers are
\[
        \frac n2,\frac n2+1,\ldots,n-1,
\]
and there are \(n/2\) of them.

If \(n\) is odd, then \((n-1)/2\) is an integer.  The value
\[
        k=\frac{n-1}{2}
\]
gives the eigenvalue
\[
        \cos\left(\frac{\pi}{2}\right)=0,
\]
which is not negative.  Hence the negative eigenvalues correspond to
\[
        k\geq \frac{n+1}{2}.
\]
There are
\[
        n-\frac{n+1}{2}
        =
        \frac{n-1}{2}
\]
such integers.  In both cases, the number of negative eigenvalues is
\[
        \left\lfloor\frac n2\right\rfloor.
\]
Therefore,
\[
        \numinus(R_n)=\left\lfloor\frac n2\right\rfloor= \mu(\Path_n)
\]

so the estimate in Theorem~\ref{thm:birth-death} is sharp for every \(n\geq2\).

For example, when \(n=4\),
\[
R_4=
\begin{pmatrix}
0 & 1 & 0 & 0\\
\frac12 & 0 & \frac12 & 0\\
0 & \frac12 & 0 & \frac12\\
0 & 0 & 1 & 0
\end{pmatrix}.
\]
The eigenvalues are
\[
        \cos(0)=1,
        \qquad
        \cos\left(\frac{\pi}{3}\right)=\frac12,
\]
\[
        \cos\left(\frac{2\pi}{3}\right)=-\frac12,
        \qquad
        \cos(\pi)=-1.
\]
Thus
\[
        \spec(R_4)=
        \left\{
        1,\frac12,-\frac12,-1
        \right\},
\]
and hence
\[
        \numinus(R_4)=2.
\]
Since
\[
        \mu(\Path_4)=2,
\]
we obtain
\[
        \numinus(R_4)=\mu(\Path_4).
\]
\end{example}
Here is another example that shows bound sharpness in the  reducible case.
\begin{example}
Let
\(
        R_2=\begin{pmatrix}0&1\\1&0\end{pmatrix}
\) be given as in the previous example.
Clearly, $ R_2$ is  stochastic, tridiagonal, and has eigenvalues $1$ and $-1$. Now 
for $n=2m$, we  define
\[
        A=\underbrace{ R_2\oplus  R_2\oplus\cdots\oplus  R_2}_{m\text{ copies}},
\]
and for $n=2m+1$, we let
\[
        A=\underbrace{ R_2\oplus  R_2\oplus\cdots\oplus  R_2}_{m\text{ copies}}\oplus (1).
\]
It is easy to see that in  both cases $A$ is  tridiagonal, and stochastic, and its spectrum consists of $m$ copies of $-1$ together with nonnegative eigenvalues. Hence,
\[
           \numinus(A)=\left\lfloor\frac n2\right\rfloor= \mu(\Path_n).
\]

\end{example}

\begin{remark} 

Obviously, the graph-theoretic bound can be substantially small 
when the support graph has small matching number. An instance where that happens, is the star graph \(G=K_{1,n-1}\) on \(n\) vertices.
This graph is connected and bipartite, and every
two edges meet at the central vertex.  Hence,
\[
        \mu(K_{1,n-1})=1.
\]
Therefore, if \(P\) is a reversible Markov transition matrix whose actual
off-diagonal support graph is \(K_{1,n-1}\), then Theorem~\ref{thm:graph-inertia}
gives
\[
        \numinus(P)\leq 1.
\]
Equivalently, \(P\) has at least \(n-1\) nonnegative eigenvalues, counted with
multiplicity.

It is worthy to note here that if we take 
\(P\) to be a nonnegative  diagonal matrix, then its off-diagonal support graph would have no edges
and would be disconnected for \(n\geq2\); in that case all eigenvalues are
nonnegative and the sharper bound \(\numinus(P)=0\) holds.  Thus, the star
example is meant to illustrate the connected, genuinely star-supported case.

\end{remark}

\subsection{Necessity of bipartiteness}

The bipartite hypothesis in Theorem~\ref{thm:graph-inertia} cannot be removed.  The smallest obstruction is an odd cycle.  Consider the simple random walk on the triangle $K_3$:
\[
P=
\begin{pmatrix}
0&\frac12&\frac12\\
\frac12&0&\frac12\\
\frac12&\frac12&0
\end{pmatrix}.
\]
This matrix is symmetric and stochastic, and its eigenvalues are given by
\[
        1,
        \quad
        -\frac12,
        \quad
        -\frac12.
\]
Therefore,
\[
        \numinus(P)=2.
\]
However,
\[
        \mu(K_3)=1.
\]
Thus,
\[
        \numinus(P)>\mu(K_3),
\]
showing that the matching-number bound in this form is genuinely bipartite.  More generally, odd cycles are the graph-theoretic obstruction to bipartiteness.

\subsection{Possible extension}
\label{rem:odd-cycle-transversal}
 Although the clean matching bound is special to bipartite graphs, one can obtain
a useful extension for an arbitrary graph by deleting a set of vertices whose
removal leaves a bipartite graph.

Let \(G=(V,E)\) be a finite graph whose adjacency matrix is $S$, and let \(U\subseteq V\).  Let \(W=V\setminus U\), and assume that
\[
        G[W]=G-U
\]
is bipartite. Then, its adjacency matrix is  the principal submatrix \(S[W]\) is obtained from \(S\) by
deleting the rows and columns indexed by \(U\).  By Cauchy's interlacing
theorem (see, e. g., \cite[Section~4.3]{HornJohnson2013}), deleting \(|U|\) rows and the corresponding columns can remove at most
\(|U|\) negative eigenvalues.  Hence
\[
        \numinus(S[W])
        \geq
        \numinus(S)-|U|.
\]
Since \(S[W]\) is supported on the bipartite graph \(G[W]\), and still has
nonnegative diagonal entries, Theorem~\ref{thm:graph-inertia} yields
\[
        \numinus(S[W])\leq \mu(G[W]).
\]
Consequently,
\[
        \numinus(S)\leq |U|+\mu(G[W]).
\]
Optimizing over all such sets \(U\), we obtain
\[
        \numinus(S)
        \leq
        \min_{\substack{U\subseteq V\\ G-U\text{ bipartite}}}
        \bigl(|U|+\mu(G-U)\bigr).
\]
Thus, the bipartite inertia estimate extends to arbitrary graphs, at the cost of
paying for the vertices that must be deleted in order to reduce the graph to a
bipartite one.

For a more informative case, consider the graph \(G\) obtained by gluing two
triangles along a single common vertex.  Let
\[
        V(G)=\{1,2,3,4,5\},
\]
and let the edge set be
\[
        E(G)=
        \{\{1,2\},\{1,3\},\{2,3\},
          \{1,4\},\{1,5\},\{4,5\}\}.
\]
Thus, \(G\) consists of the two triangles

sharing the vertex \(1\).

The set
\[
        U=\{1\}
\]
meets every odd cycle of \(G\).  After deleting \(U\), the remaining graph is
\[
        G-U
\]
with vertex set \(\{2,3,4,5\}\) and edges
\[
        \{2,3\},
        \qquad
        \{4,5\}.
\]
Hence, \(G-U\) is bipartite, and
\[
        \mu(G-U)=2,
\]
so that 
\[
        \numinus(S)
        \leq
        |U|+\mu(G-U)
        =
        1+2
        =
        3.
\]
This is a genuinely nontrivial bound for a \(5\times5\) matrix, since the trace
condition alone would only rule out five negative eigenvalues.

Moreover, the bound is sharp in this case.  Indeed, take \(S\) to be the
adjacency matrix of \(G\):
\[
S=
\begin{pmatrix}
0&1&1&1&1\\
1&0&1&0&0\\
1&1&0&0&0\\
1&0&0&0&1\\
1&0&0&1&0
\end{pmatrix}.
\]
This matrix is real symmetric, supported on \(G\), and has nonnegative diagonal.
A direct computation gives
\[
        \operatorname{spec}(S)
        =
        \left\{
        \frac{1+\sqrt{17}}{2},
        1,
        -1,
        -1,
        \frac{1-\sqrt{17}}{2}
        \right\}.
\]
Thus, \(S\) has exactly three negative eigenvalues:
\[
        -1,\qquad -1,\qquad \frac{1-\sqrt{17}}{2}.
\]
Consequently,
\[
        \numinus(S)=3,
\]
which agrees exactly with the bound
\[
        |U|+\mu(G-U)=3.
\]
These examples illustrate the role of \(U\): each vertex in \(U\) is paid for
directly, while the remaining bipartite part is controlled by its matching
number.

 Finally, we close with the  following natural questions.
\begin{question} Under what additional conditions,  can we generalize the results here to  k-banded stochastic matrices?
\end{question}
\begin{question}
How can the present work be generalized to  block-tridiagonal stochastic matrices?
\end{question}

\end{document}